\documentclass[12pt]{amsart}
\usepackage{geometry,graphicx}
\usepackage{float}
\usepackage{amssymb,enumerate}
\usepackage{amsmath,latexsym}

\usepackage{hyperref} 

\newcommand{\R}{{\mathbb R}}

\newcommand{\N}{{\mathbb N}}

\newcommand{\Z}{{\mathbb Z}}

\newcommand{\Q}{{\mathbb Q}}

\newcommand{\diam}{{\rm diam}\,}

\newcommand{\Norm}{{\| \cdot \|}}
\newcommand{\vece}{{\bf e}}

\newcommand{\Syst}{{\operatorname{Syst}}}
\newcommand{\mG}{{\mathcal{G}}}
\newcommand{\Image}{{\operatorname{Im}}}
\newcommand{\mT}{{\mathcal{T}}}

\newcommand{\mF}{\mathcal{F}}
\newcommand{\mC}{{\mathcal{C}}}

\newcommand{\Lcor}{{L_{\rho_1}^{\rm{cor}}}}
\newcommand{\Lhub}{{L_{\rho_1}^{\rm{hub}}}}

\numberwithin{equation}{section} 
\newtheorem{theorem}[equation]{Theorem}
\newtheorem{lemma}[equation]{Lemma}
\newtheorem{corollary}[equation]{Corollary}
\newtheorem{proposition}[equation]{Proposition}
\newtheorem{remark}[equation]{Remark}
\newtheorem{dfn}[equation]{Definition}

\newtheorem{notarem}[equation]{Notation, Assumptions \& Remarks}

\begingroup\makeatletter\ifx\SetFigFont\undefined%
\gdef\SetFigFont#1#2#3#4#5{%
 \reset@font\fontsize{#1}{#2pt}%
 \fontfamily{#3}\fontseries{#4}\fontshape{#5}%
 \selectfont}%
\fi\makeatother\endgroup%

\begin{document}

\title[metrics on a $2$-torus with a partially prescribed stable norm]{Constructing metrics on a $2$-torus with a partially prescribed stable norm}


\author[E. Makover]{Eran Makover}
\address[Eran Makover]{Central Connecticut State University\\ Department of Mathematics \\New Britain, CT}
\email{makover@ccsu.edu}


\author[H. Parlier]{Hugo Parlier $^\ddagger$}
\address[Hugo Parlier]{Department of Mathematics, University of Fribourg\\
  Switzerland}
\email{hugo.parlier@gmail.com}\thanks{$^\ddagger$Research partially supported by Swiss National Science Foundation grant number PP00P2\textunderscore 128557}


\author[C. J. Sutton]{Craig J. Sutton $^\sharp$}
\address[Craig J. Sutton]{Dartmouth College\\ Department of Mathematics \\ Hanover, NH 03755}
\email{craig.j.sutton@dartmouth.edu}
\thanks{$^\sharp$Research partially supported by National Science Foundation
grant DMS 0605247 and a Career Enhancement Fellowship from the 
Woodroow Wilson National Fellowship Foundation}



\begin{abstract}
A result of Bangert states that the stable norm associated to any Riemannian metric 
on the $2$-torus $T^2$ is strictly convex. 
We demonstrate that the space of stable norms associated to metrics on $T^2$ 
forms a proper dense subset of the space of strictly convex norms on $\R^2$.
In particular, given a strictly convex norm $\Norm_\infty$ on $\R^2$ we construct 
a sequence $\langle \Norm_j \rangle_{j=1}^{\infty}$ of stable norms that converge 
to $\Norm_\infty$ in the topology of compact convergence
and have the property that for each $r > 0$ there is an $N \equiv N(r)$ such that $\Norm_j$ 
agrees with $\Norm_\infty$ on $\Z^2 \cap \{ (a,b) : a^2 + b^2 \leq r \}$ for all $j \geq N$. 
Using this result, we are able to derive results on multiplicities which arise in the minimum 
length spectrum of $2$-tori and in the simple length spectrum of hyperbolic tori.
\end{abstract}
\maketitle

\section{Introduction}

Given a closed $n$-dimensional manifold $M$ with first Betti-number $b = b_1(M)$, 
we let $H_1(M; \Z)_\R$ denote the collection of integral classes in 
the $b$-dimensional real vector space $H_1(M; \R)$. 
Then $H_1(M;\Z)_\R$ is a co-compact lattice in $H_1(M; \R)$.
Letting $T \simeq \Z_{m_1} \times \cdots \times \Z_{m_q}$ denote the torsion subgroup of 
$H_1(M; \Z) \simeq \Z^b \times T$, we see that $H_1(M; \Z)_\R$ 
can be identified with $H_1(M; \Z)/T$ via the surjective homomorphism $\phi : H_1(M; \Z) \to H_1(M; \Z)_\R$
given by 
$$\sum_{i=1}^{b} z_i h_i + t \mapsto (\sum_{i=1}^{b} z_i h_i) \otimes_{\Z} 1,$$
where $\{h_1, \ldots , h_b\}$ is some $\Z$-basis for $H_1(M; \Z)$, the $z_i$'s are integers and $t \in T$.
Now, let $\Psi: \pi_1(M) \to H_1(M; \Z)$ denote the Hurewicz homomorphism \cite{Lee}, then the 
regular covering $p_{\rm{Abel}}: M_{\rm{Abel}} \to M$ of $M$ corresponding to $\ker(\Psi) = [\pi_1(M), \pi_1(M)]$ 
is the \emph{universal abelian covering} of $M$. It is universal in the sense that 
it covers any other normal covering for which the deck transformations form an abelian group. 
The \emph{universal torsion-free abelian cover} $p_{\rm{tor}} : M_{\rm{tor}} \to M$ corresponds to the normal subgroup 
$\Psi^{-1}(T) \lhd \pi_1(M)$: it covers all other normal coverings for which the group of deck transformations 
is torsion-free and abelian. Under the above identifications we see that the group of deck transformations of $M_{\rm{tor}} \to M$ 
is given by the lattice $H_1(M; \Z)_\R$. If $M$ has positive first Betti number, then to each metric $g$  
we may associate a geometrically significant norm $\| \cdot \|_s$ on $H_1(M; \R)$ in the following manner.

For each $h \in H_1(M; \Z)_\R \simeq \Z^b \leq H_1(M; \R)$ let 
$$f(h) = \inf \{ L_g(\sigma) : \sigma \mbox{ is a smooth loop representing the class } h \},$$
where $L_g$ is the length functional associated to the Riemannian metric $g$ on $M$. Then for each $n \in \N$ 
we let $f_n : \frac{1}{n}H_1(M; \Z)_\R \to \R_{\geq 0}$ be given by 
$$f_n(h) = \frac{1}{n}f(nh).$$
It can be seen that the $f_n$'s converge uniformly on comapct sets to a norm $\| \cdot \|_s$ on $H_1(M;\R)$ that is 
known as the \emph{stable norm} of $g$ \cite{Bang}. 
In particular, if $\{ v_n \}_{n \in \N}$ is a sequence in $H_1(M; \Z)_\R$ 
such that $\lim_{n \to \infty} \frac{v_n}{n} = v \in H_1(M; \R)$, then 
$$\| v \|_s = \lim_{n \to \infty} \frac{f(v_n)}{n}.$$
An integral class $v \in H_1(M; \Z)_{\R}$ is said to be \emph{stable} if there is an $n \in \N$ such that 
$\| v \|_s = f_n(v) = \frac{f(nv)}{n}.$
Intuitively, the stable norm $\Norm_s$ describes the geometry of the universal torsion-free abelian cover
$(M_{\rm{tor}}, g_{\rm{tor}})$ in a manner where the fundamental domain of 
the $H_1(M; \Z)_\R$-action appears to be arbitrarily small (cf. \cite[p. 250]{Gromov}).

Now, let $p: (N, h) \to (M,g)$ be a Riemannian covering. We will say that a non-constant geodesic 
$\gamma : \R \to (M,g)$ is $p$-\emph{minimal} (or \emph{minimal with respect to $p$}) 
if for some and, hence, every lift $\tilde{\gamma}: \R \to N$ of $\gamma$, the geodesic $\tilde{\gamma}$ is distance minimizing between any two of its points. That is, $\gamma$ is $p$-minimal if for any $t_1 \leq t_2$ we have $d_N(\tilde{\gamma}(t_1), \tilde{\gamma}(t_2)) = L_g(\tilde{\gamma} \upharpoonright [t_1, t_2])$.
In the event that $p$ is the universal Riemannian covering we will refer to $p$-minimal geodesics as \emph{minimal}, and when $\gamma$ is minimal with respect to the universal abelian cover 
$p_{\rm{abel}}: (M_{\rm{Abel}}, h) \to (M,g)$ we will say that $\gamma$ is 
an \emph{abelian minimal geodesic}. In the case where $\pi_1(M)$ is abelian---e.g., $M$ is a torus---these two definitions coincide.

An interesting application of the stable norm $\Norm_{s}$ is that characteristics of its unit ball 
$B \subset H_1(T^2; \Z)$ can be used to deduce the existence (and properties) of minimal 
abelian geodesics. For instance, we have the following result due to Bangert.

\begin{theorem}[\cite{Bang} Theorem 4.4 \& 4.8]
Let $(M,g)$ be a Riemannian manifold and let 
$B \subseteq H_1(M;\R)$ be the unit ball corresponding to its stable norm. 
For every supporting hyperplane $H$ of $B$ there is an abelian minimal geodesic
$\gamma: \R \to (M,g)$. As a consequence, $(M,g)$ has at least 
$k \equiv \dim H_1(M;\R)$ geometrically distinct abelian minimal geodesics. 
\end{theorem}

\noindent
In light of the relationship between the existence of minimal geodesics and the unit ball of the stable norm, 
it is an interesting question to determine which norms on $H_1(M; \R)$ arise as the stable norm associated 
to a Riemannian metric on $M$. In the case of the two-torus Bangert has made the following observation.

\begin{theorem}[\cite{Bang} p. 267, \cite{Bang2} Sec. 5]\label{thm:StrictlyConvex}
The collection of stable norms on $T^2$, denoted $\mathcal{N}_{\rm{stab}}(T^2)$, 
is a proper subset of the collection of strictly convex norms on $\R^2$, denoted 
by $\mathcal{N}_+(\R^2)$.
\end{theorem}

Indeed, for any metric $g$ on a $2$-torus we have that $f(kh) = |k|f(h)$ for any $k \in \Z$ and $h \in H_1(M; \Z)_{\R}$.
Therefore, $\|h\|_s = f(h)$ on $H_1(M; \Z)_{\R}$. Now, suppose $h_1, h_2 \in H_1(M; \Z)_{\R}$ are rationally independent 
and are represented by shortest geodesics $\gamma_1$ and $\gamma_2$ respectively. Then $\gamma_1$ and $\gamma_2$ must intersect transversally at $\gamma_1(0) = \gamma_2(0)$, for instance, 
and we conclude that $\gamma_1 * \gamma_2$ is not smooth.
Therefore, since  the non-smooth curve $\gamma_1 * \gamma_2$ represents the integral  
homology class $h_1 + h_2$ we obtain the following strict inequality
$$\| h_1 + h_2\|_s < L_g(\gamma_1* \gamma_2) = L_g(\gamma_1) + L_g(\gamma_2) = \|h_1\|_s + \|h_2\|_s.$$
It then follows that $\Norm_s$ is strictly convex norm on $H_1(T^2, \R) \simeq \R^2$. 
To see that $\mathcal{N}_{\rm{stab}}(T^2)$ is a proper subset of $\mathcal{N}_+(\R^2)$, we recall that Bangert observed that on $T^2$ the stable norm is differentiable at  irrational points \cite[Sec. 5]{Bang2}. 
That is, the unit ball of a stable norm associated to a Riemannian metric on $T^2$ has a unique 
supporting line at points $(x,y)$ where $y/x$ is irrational. 
But, one can readily see that there are many strictly convex norms which are not differentiable at such points. 
For instance, one need only take a strictly convex norm for which the 
unit ball is a tear drop whose singularity is placed at $(x,y)$ with $y/x$ irrational. 
And we conclude that $\mathcal{N}_{\rm{stab}}(T^2)$ is a proper subset of $\mathcal{N}_+(\R^2)$.

In this article we will be concerned with stable norms of Riemannian $2$-tori; 
henceforth referred to as \emph{toral stable norms}.
We show that the toral stable norms form a dense proper subset in the 
collection of all strictly convex norms on  $H_1(T^2; \R)\simeq \R^2$. 
Specifically, we demonstrate the following.

\begin{theorem}\label{thm:StableNorm}
Let $\Norm_\infty$ be a strictly convex norm on $H_1(T^2; \R)$ and 
let $\langle h_j \equiv (a_j, b_j)\rangle_{j =1}^{\infty}$ be a sequence consisting of all of the 
integral homology classes $H_1(M; \Z)_{\R} \simeq \Z^2$ where $\| (a_j, b_j) \|_\infty \leq \| (a_{j+1}, b_{j+1}) \|_\infty$ for each $j$. 
Then there exists a sequence $\langle \Norm_j \rangle_{j = 1}^{\infty}$ of toral stable norms such that
\begin{enumerate}[{\bf(i)}] 
\item for each $k \in \N$ we have
$\| (a_j, b_j) \|_k = \| (a_j, b_j) \|_\infty \mbox{ for } 1\leq j \leq k,$
while $\| (a_j, b_j) \|_k \geq \| (a_k, b_k) \|_\infty$ for all $j \geq k+1$; 
\item $\lim_{j \to \infty} \Norm_j = \Norm_\infty$ in the topology of compact convergence.
\end{enumerate}
\end{theorem}

\noindent
Hence, any strictly convex norm on $\R^2$ can be approximated uniformly on compact sets by a stable norm 
that agrees with it on an arbitrarily large set of lines through the origin with rational slope. 
We now show that this result can be interpreted in terms of the minimum marked length spectrum of a torus.

First, we recall that the \emph{length spectrum} of a Riemannian manifold $(M,g)$ is the collection 
of lengths of all smoothly closed geodesics in $(M,g)$, where we adopt the convention that the multiplicity of 
a length $\ell$ is counted according to the number of free homotopy classes containing a 
geodesic of that length. 
Now, given a loop $\sigma$ on a manifold $M$ its \emph{unoriented free homotopy class} 
is the collection of closed geodesics that are freely homotopic to $\sigma$ or 
its inverse $\overline{\sigma}$. 
We will denote the collection of the unoriented free homotopy classes by $\mathcal{F}(M)$
and let $\pi: \pi_1(M) \to \mF(M)$ denote the natural projection\label{def:NatProj}. 
We then define the \emph{minimum length spectrum} to be the (possibly finite) sequence 
$\ell_1 = 0 < \ell_1\leq \ell_2 \leq \cdots $ consisting of the lengths of closed geodesics 
that are shortest in their \emph{unoriented} free homotopy class, where a length $\ell$ is 
repeated according to the number of unoriented free homotopy classes whose shortest 
geodesic is of length $\ell$.  
If we wish to keep track of the unoriented free homotopy classes we then 
consider the map $m_g: \mathcal{F}(M) \to \R$ which assigns to each unoriented 
free homotopy class the length of its shortest closed geodesic.
We will refer to $m_g$ or the collection $\{(m_g(\alpha), \alpha): \alpha \in \mathcal{F}(M)\}$ as 
the \emph{minimum marked length spectrum} of $(M,g)$ (see \cite[Def. 2.8]{DGS}).

It is natural to ask which pairs $(\ell, \alpha)$ consisting of a nonnegative number $\ell$ and 
an unoriented free homotopy class $\alpha$ can occur as part of the mimium marked length spectrum 
associated to some metric $g$ on $M$. This question was addressed in dimension three and higher 
by the third author, along with De Smit and Gornet, in \cite{DGS} where the following was shown.

\begin{theorem}[\cite{DGS} Theorem 2.9]\label{thm:DGS}
Suppose that $M$ is a closed connected manifold of dimension at least three.
Let $\alpha = ( \alpha_1, \alpha_2, \ldots, \alpha_k )$ be 
a sequence of
distinct elements of $\mF(M)$ where $\alpha_1$ is trivial.
Then for every sequence $0 = \ell_1 < \ell_2 \leq \cdots \leq \ell_k$
of real numbers the following are equivalent:
\begin{enumerate}[{\bf (i)}]
\item The sequence $\ell_1,\ldots, \ell_k$ is $\alpha$-admissible; 
that is, for $i, j = 2, \ldots , k$ $\ell_i \leq |n| \ell_j$, whenever
$\alpha_i = \alpha_j^n$ for some $n \in \Z$ and for $i = 2, \ldots , k$ $\ell_i \geq \frac{1}{|n|} \ell_k$ whenever 
$\alpha_i^n \not\in \{\alpha_1, \ldots , \alpha_k \}$ for some $n \neq 0 \in \Z$.
\item There is a Riemannian metric $g$ on $M$ such that 
the minimum marked length spectrum $m_g : \mF(M) \to \R_{\geq 0}$ satisfies 
$m_g(\alpha_i) = \ell_i$ for all $i$ and $m_g(\alpha) \geq \ell_k$ for all 
$\alpha \in \mF(M)-\{\alpha_1, \ldots , \alpha_k \}$.

\end{enumerate}
In particular, there is a metric $g$ on $M$ such that the systole is achieved
in the unoriented free homotopy class $\alpha_2$. 
\end{theorem}

\noindent

The proof of Theorem~\ref{thm:DGS} depends on the fact that a finite collection of 
distinct unoriented free homotopy classes can be represented by pairwise disjoint simple closed curves. 
The fact that this does not hold in dimension two appears to make approaching 
this question for surfaces---the actual motivation behind this article---a more delicate matter. 
However, we note that among surfaces the torus enjoys some special properties. 
First, all free homotopy classes can be represented by a simple closed curve or an iterate of such a curve. 
Consequently, with respect to any metric, the shortest closed geodesic in a free homotopy class will be a simple 
closed curve if the class is primitive, or an iterate of a simple closed curve in the case of a non-primitive class. 
Secondly, it follows from the fact that $T^2$ is an aspherical surface that 
for any choice of smooth Riemannian metric $g$ and choice of non-trivial free homotopy class $[\beta]$, 
a closed geodesic of minimal length in $[\beta]$ will have a minimal number of self-intersections \cite{FHS}. 
In Section~\ref{sec:Construction}, these properties will be marshaled to prove 
Theorem~\ref{thm:StableNorm}{\bf(i)} which in conjunction with Bangert's Theorem~\ref{thm:StrictlyConvex}
gives the following statement concerning the minimum marked length spectrum of a $2$-torus.

\begin{theorem}\label{thm:MinLengthSpec}
Let $T^2$ be a $2$-torus and $\pi_1(T^2) \simeq \Z^2 \leq \R^2$ its fundamental group.
Now, let $\alpha= (\alpha_1, \ldots , \alpha_k)$ be a sequence of distinct unoriented free 
homotopy classes of $T^2$, where $\alpha_i$ is represented by $\pm (a_i,b_i) \in \Z^2$ and 
$\alpha_1 = (0,0)$ is trivial. Also, let 
$\ell_1 = 0 < \ell_2 \leq \cdots \leq \ell_k$ be a finite sequence. Then the following are equivalent:
\begin{enumerate}[{\bf (i)}]
\item There is a strictly convex norm $\| \cdot \|$ on $\R^2$ such that $\| (a_i, b_i) \| = \ell_i$ and 
$\|(a,b) \| \geq \ell_k$ for any $(a,b) \neq \pm ( a_1, b_1), \ldots , \pm (a_k, b_k)$.
\item There is a metric $g$ on $T^2$ such that the 
minimum marked length spectrum $m_g: \mF(T^2) \to \R_{\geq 0}$ satisfies 
$m_g(\alpha_i) = \ell_i$ for all $i = 1, \ldots, k$ and $m_g(\alpha) \geq \ell_k$ 
for all $\alpha \in \mF(T^2) - \{\alpha_1, \ldots , \alpha_k\}$. 
\end{enumerate}
\end{theorem}

In Section~\ref{sec:Multiplicity} we consider the multiplicities in the minimum length spectrum of a $2$-torus.
By using results  concerning the minimum number of lattice points in the interior of an $n$-gon
and Theorem~\ref{thm:StableNorm}{\bf(i)} we obtain the following estimate on the ``location'' of a length 
with a specified multiplicity.

\begin{theorem}\label{thm:Multiplicity}
Suppose $(T^2, g)$ is a torus for which the minimum length spectrum $\langle \ell_j \rangle_{j=1}^{\infty}$
 has a length of multiplicity $m$. That is, for some $n \in \N$ we have $0 = \ell_1 \leq \ell_n < \ell_{n+1} = \cdots = \ell_{n+m} < \ell_{n+m+1}$. 
Then $n = m_g^{-1}([0, \ell))= f(m) \equiv \frac{i_0^{\rm{symm}}(2m) +1}{2} \geq O(m^3)$, where 
$i_0^{\rm{symm}}(2m)$ is the minimum number of integer points in the interior of 
a convex integer $2m$-gon that is centrally symmetric with respect to $(0,0)$.
Furthermore, this inequality is sharp. That is, for each $m \in \N$ there is a smooth metric $g$ on $T^2$
and $\ell > 0$ such that $\ell$ has multiplicity $m$ in the minimum length spectrum and 
$m_g^{-1}([0,\ell)) = f(m)$. 
\end{theorem}

Our study of the multiplicities of the minimum length spectrum of a torus  
is motivated in part by the study of hyperbolic surfaces; especially, hyperbolic punctured tori. The length spectrum of a hyperbolic surface always contains lengths of arbitrarily high multiplicity \cite{Ran}, and any closed geodesic is of minimal length on a hyperbolic surface. Unlike the case of smooth tori, hyperbolic surfaces contain non-simple closed geodesics which are thus minimal in their homotopy class, and it is among these geodesics that high multiplicities are known to appear. To date, multiplicities have not been observed among the simple closed geodesics and it is a conjecture of Schmutz Schaller that among primitive simple closed geodesics on a once-punctured torus 
the multiplicity of a given length is bounded by $6$. 
This conjecture is a specific case of a more general conjecture, due to Rivin, 
asserting that multiplicity in the \emph{simple length spectrum}---the collection 
of lengths of simple closed geodesics---is always bounded by a constant 
that only depends on the underlying topology (see \cite[p. 209]{Schmutz}).

Presently, not much is known about the validity of the conjectures of Schmutz Schaller and Rivin. However, Theorem~\ref{thm:Multiplicity} gives new examples demonstrating that these conjectures do not hold for arbitrary surfaces; in particular, tori (cf. \cite[p. 1884-5]{MP}). We note that Theorem~\ref{thm:Multiplicity} can be used to relate the multiplicity of the length $\ell$ to its position in the simple length spectrum of a one-holed or once-punctured torus.

\begin{corollary}\label{cor:HypMult} If there are $m$ simple closed geodesics of the 
same length $\ell$ on a once-punctured (or one-holed) torus, then there are at least 
$f(m)$ distinct simple closed geodesics of length strictly less than $\ell$.
\end{corollary}

\noindent
Unlike the conjectures of Schmutz Schaller and Rivin, the geodesics considered in Corollary~\ref{cor:HypMult}
include geodesics representing non-primitive classes and the function $f(m)$ counts the trivial 
homology/homotopy class.
Of course if Rivin's conjecture is correct, then Corollary~\ref{cor:HypMult} might only 
be of interest for small values of $m$.

\section{Constructing the Stable Norms: the Proofs of 
Theorems~\ref{thm:StableNorm}{\bf (i)} and \ref{thm:MinLengthSpec} }\label{sec:Construction}

In this section we will prove Theorems~\ref{thm:StableNorm}{\bf(i)} and \ref{thm:MinLengthSpec}.
The basic idea behind the proof of Theorem~\ref{thm:StableNorm}{\bf(i)} is 
to isolate geodesics $\gamma_1, \gamma_2, \ldots , \gamma_k$ 
on a flat torus $(T^2, g_0)$, with a systole of at least $\ell_k$, 
representing the $k$ homology classes $h_1, h_2, \ldots , h_k \in H_1(T^2,\Z)$ in the 
statement of the Theorem and then dig deep ``canyons'' with narrow ``corridors'' of the appropriate length 
along these geodesics in order to obtain a new metric $g_k$ for 
which the conclusions of the theorem are obtained. 
Theorem~\ref{thm:MinLengthSpec} will then follow as an application of 
Theorem~\ref{thm:StableNorm}{\bf(i)} and Bangert's result that 
the stable norm of a metric on a $2$-torus is strictly convex. 

\subsection{The proof of Theorem~\ref{thm:StableNorm}(i)} 
Let $\Norm_\infty$ denote a fixed strictly convex norm on $H_1(T^2, \R) \simeq \R^2$ and 
let $\langle h_i = (a_i, b_i) \rangle_{i = 1}^{\infty}$ denote a fixed enumeration of the integral 
homology classes $H_1(T^2, \Z) \simeq \Z^2$ with the property that $\|h_i\|_\infty \leq \|h_{i+1}\|_\infty$ for each $i \in \N$.
In this section we wish to show that for each $k \in \N$ we may find a toral stable norm $\Norm_k$ such that 
$\|h_i\|_k = \|h_i\|_\infty$ for $1\leq i \leq k$, while $\|h_j\|_k \geq \|h_k\|_\infty$ for each $j \geq k+1$. 
We begin by fixing some notation and assumptions that will hold throughout this section.\\

\begin{notarem}\label{notarem:Assumptions} \text{}
\begin{enumerate}[1.]
\item For any Riemannian metric $g$ on $T^2$ we let $L_g$ denote the length functional on the loop space and 
we let $d_g$ be the distance function in the induced metric space structure.

\item For any loop $\sigma: S^1 \to T^2$ we will let $\Image(\sigma)$ denote the image of $\sigma$ 
and we wil let $h_\sigma \equiv (a_\sigma, b_\sigma) \in H_1(T^2; \Z) \simeq \Z^2$ denote its homology class.

\item For each $i \in \N$ we will let $\ell_i = \|h_i\|_\infty$.

\item We will say that a homology class $h \in H_1(T^2; \Z)$ is \emph{primitive} if whenever $h = n \widetilde{h}$, for some $n \in \N$ and $\widetilde{h} \in H_1(T^2; \Z)$, we have $n =1$ and $\widetilde{h} = h$.

\item Since for any norm $\Norm$ on a real vector space $\mathcal{V}$ we have $\|rv\| = |r|\|v\|$, 
where $r \in \R$ and $v\in \mathcal{V}$, we may assume without loss of generality that 
each $h_i =(a_i, b_i) \in H_1(T^2; \Z)$ is a primitive homology class and that for $i \neq j$ we have $h_i \neq \pm h_j$. 

\item When convenient we will identify a homology class $h \in H_1(T^2; \Z)$ with 
the free homotopy class $\Psi^{-1}(h)$ given by the Hurewicz isomorphism $\Psi: \pi_1(T^2, p_o) \to H_1(T^2; \Z)$, where $p_0$ is some fixed point in $T^2$.

\item We will let $g_0$ denote a fixed \emph{flat} metric on $T^2$ with systole satisfying $\Syst(T^2, g_0) \geq \ell_k$ and set $B \equiv \Syst(T^2, g_0)$.

\item For each $i \in \N$ we will let $\gamma_i$ be the unique geodesic in $(T^2, g_0)$ passing through $p_0$ and representing the primitive homology class $h_i$. We note that since $h_1 = (0,0)$ represents the trivial class, the geodesics $\gamma_1$ is trivial.

\item\label{thm:Rephrased} Theorem~\ref{thm:StableNorm}{\bf(i)} is then equivalent to showing that for each $k \in \N$ there is a metric $g_k$
such that 
\begin{enumerate}
\item $L_{g_k}(\gamma_i) = \ell_i$,
\item for any loop $\sigma$ in $(T^2, g_k)$, representing one of the (primitive) homology classes $\{h_i\}_{i \in \N}$, we have 
$$L_{g_k}(\sigma) \geq \left\{ \begin{array}{ll}
\ell_i & h_\sigma = h_i \; \mbox{for some } i = 1, \ldots , k\\
\ell_k & \mbox{otherwise}
\end{array}
\right.
$$
\end{enumerate}

\item By a \emph{cycle} $c$ in a graph $\mG$ we will mean a sequence of vertices $\langle v_i \rangle_{i=0}^{q}$ such that 
$v_0 = v_q$ and for each $i = 0, 1, \ldots , q-1$ there is an edge $e_i$ joining $v_i$ and $v_{i+1}$.
The edge length of such a cycle is said to be $q$.

\item It is clear that if $(\mG, d)$ is a metric graph, then for any loop $\sigma : S^1 \to \mG$ 
there is a cycle $c$ that is freely homotopic to $\sigma$ in $\mG$ such that $L(\sigma) \geq L(c)$.
A cycle $c$ will be said to be \emph{minimal} if it is the shortest cycle in its free homotopy class. Clearly a minimal cycle will have minimal edge length among all other cycles in its free homotopy class. 

\end{enumerate}
\end{notarem}

\text{}\\

Fix $k \in \N$ and let $h_1= (a_1, b_1), \ldots , h_k = (a_k, b_k) \in H_1(T^2; \Z)$ be the first 
$k$ homology classes in our ordering.  
Since $T^2$ is a torus we see that for each $2 \leq i \neq j \leq k$ the geodesics $\gamma_i$ and $\gamma_j$ intersect transversally in 
finitely many points.
Consider the curves $\gamma_1, \gamma_2, \ldots , \gamma_k$ simultaneously 
and let $\{p_0, p_1, \ldots , p_t\}$ be the collection of intersection points. 
Then for each $i = 2, \ldots , k$ these points partition $\gamma_i$ into $m_i$ segments 
$\gamma_{i1}, \ldots , \gamma_{im_i}$, and since $g_0$ is a flat metric on $T^2$
one can deduce that the quantity $q_{ij} \equiv \frac{L_{g_0}(\gamma_{ij})}{L_{g_0}(\gamma_i)}$ 
is a positive rational number, for each $i = 2, \ldots , k$ and $1 \leq j \leq m_i$.
The union of the images of the geodesics $\gamma_1, \gamma_2, \ldots , \gamma_k$, which we 
will denote by $\mG$, forms a directed graph in $T^2$, where the points 
$\{p_0, p_1, \ldots , p_t\}$ are the vertices and the segments $\gamma_{ij}$ are the oriented edges.
Now suppose $\mT$ is a regular neighborhood of $\mG$ with 
smooth boundary in $T^2$ (see Figure~\ref{Fig:Tubular}). 
Then $\mT$ can be decomposed into $t+1$ disjoint ``hubs'' 
$\{\Delta_0, \Delta_1, \ldots , \Delta_t \}$ containing the vertices $\{p_0, p_1, \ldots , p_t \}$ and disjoint (rectangular) ``corridors'' $R_{ij}$ containing $\Image(\gamma_{ij}) - \cup_{s = 0}^{t} \Delta_s$ (see Figure~\ref{Fig:NGons}).
We now show that we can find a regular neighborhood $\mT$ of $\mG$ and 
a flat metric $\rho_1$ defined on $\mT$ such that Theorem~\ref{thm:StableNorm}{\bf (i)}---in the guise of \ref{notarem:Assumptions}(\ref{thm:Rephrased}) above---is true if we restrict our attention to loops 
contained in $(\mT, \rho_1)$. Specifically, we have the following lemma.

\begin{figure}
\includegraphics[width=2in]{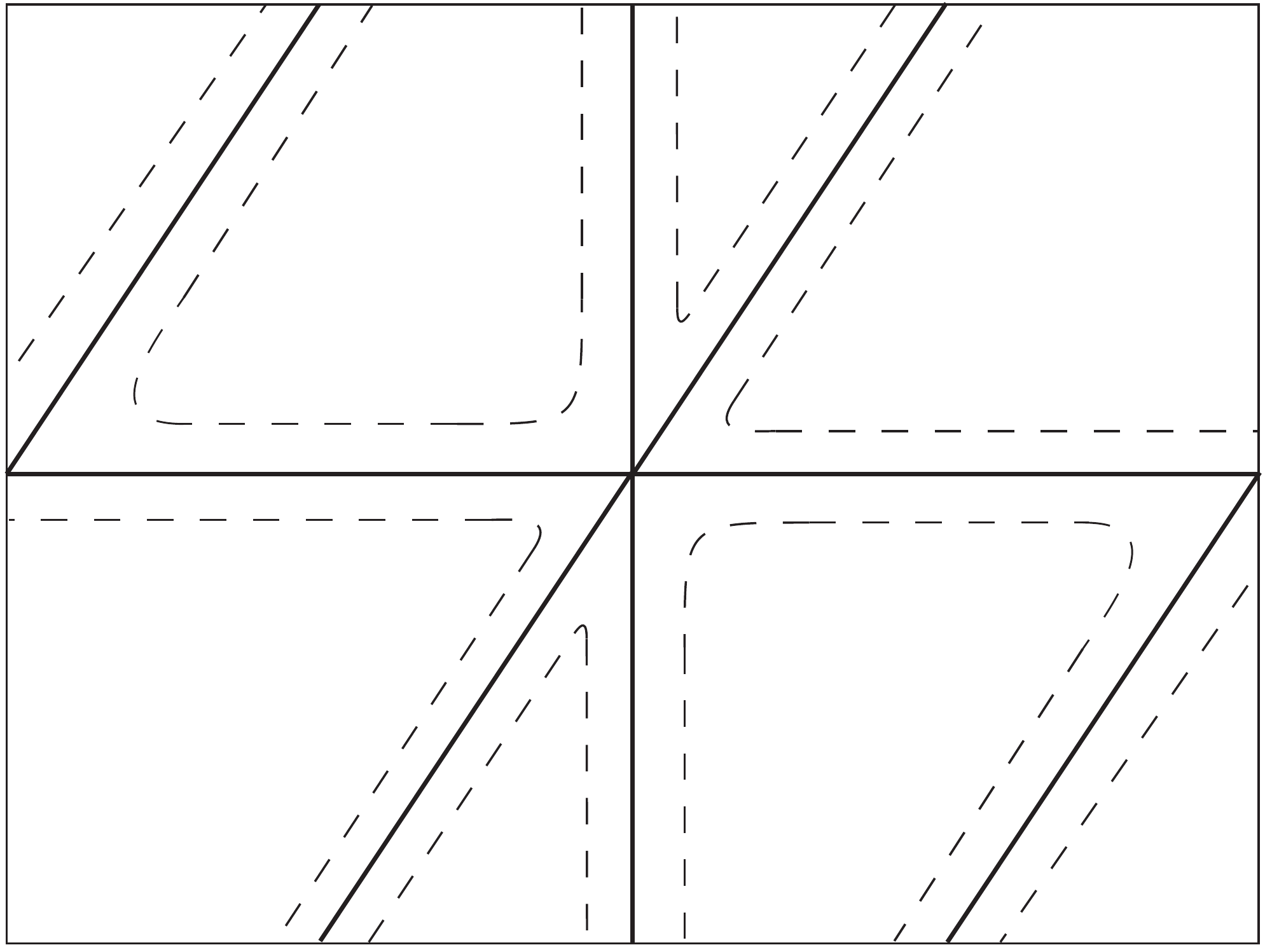}
\caption{Regular Neighborhood of $\mG$}
\label{Fig:Tubular}
\end{figure}

\begin{figure}
\includegraphics[width=2in]{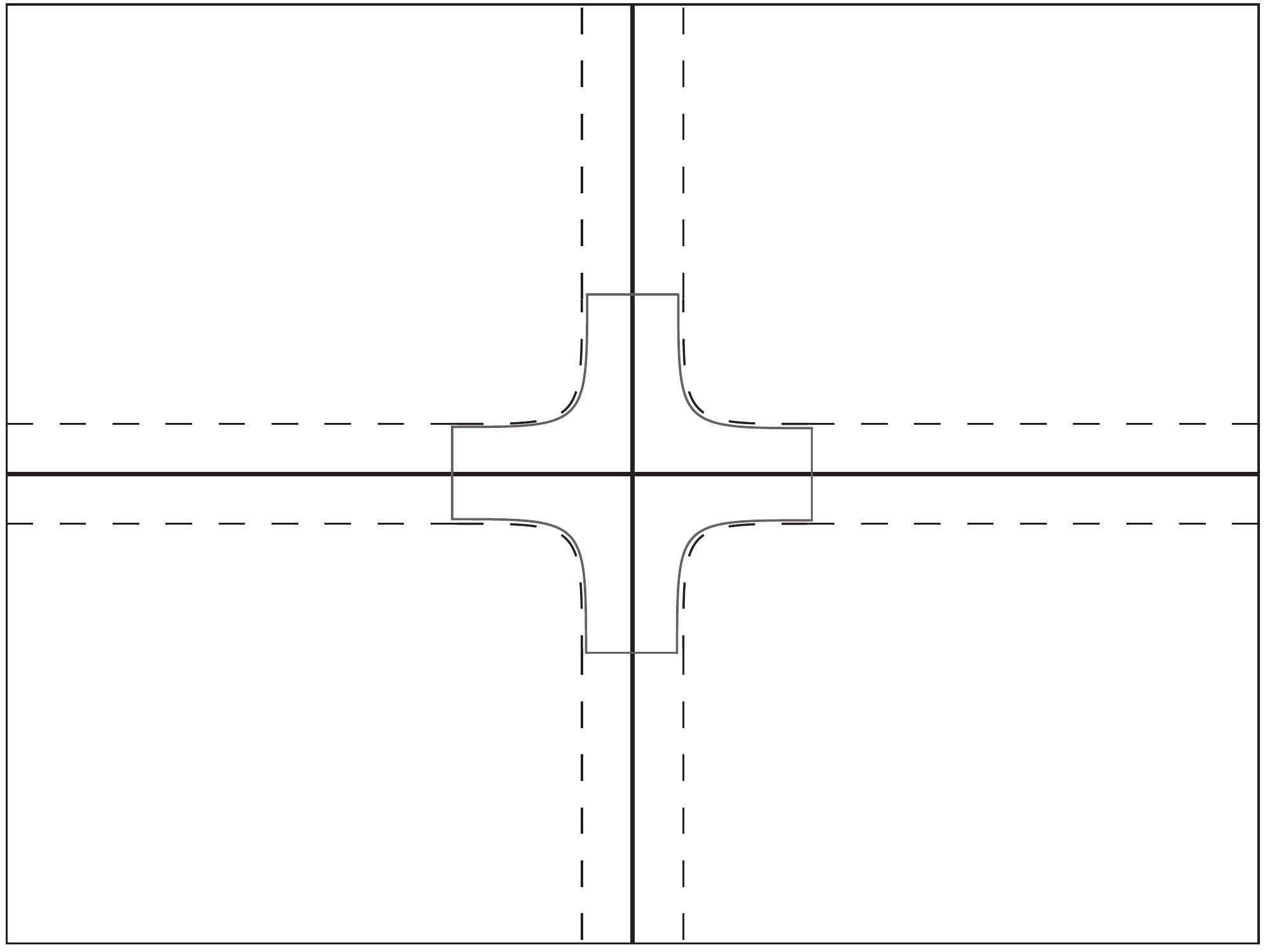}
\caption{Decomposition into ``hubs'' \& ``corridors''}
\label{Fig:NGons}
\end{figure}

\begin{lemma}\label{lem:TubNbhd}
With the notation as above, there is a regular neighborhood $\mT$ of $\mG$ 
with smooth boundary and a flat metric $\rho_1$ on $\mT$ with the following properties:
\begin{enumerate}[{\bf (1)}]
\item $L_{\rho_1}(\gamma_i) = \ell_i$ for $i = 1,2, \ldots , k$; 
\item if $\sigma$ is a loop in $\mT$ representing the (primitive) homology class 
$(a_\sigma, b_\sigma) \in H_1(T^2, \Z) \simeq \Z^2$, then 
$$L_{\rho_1}(\sigma) \geq \left\{ \begin{array}{ll}
\ell_i & (a_\sigma, b_\sigma) = \pm (a_i, b_i) \; \mbox{for some } i = 1, \ldots , k\\
\ell_k & \mbox{otherwise}
\end{array}
\right.
$$
\end{enumerate}

\end{lemma}

\begin{proof}[Proof of Lemma~\ref{lem:TubNbhd}]

The proof of this lemma has three main steps:

\begin{enumerate}[(A)]
\item We take an arbitrary regular neighborhood $\mT'$ of $\mG$ with a particular choice of ``hubs'' $\{ \Delta_0, \Delta_1, \ldots , \Delta_t\}$ and corresponding rectangular ``corridors'' $\{R_{ij} : i = 2, \ldots , k, j = 1, \ldots , m_i \}$.
Then we adjust the length of the corridors to obtain a flat metric $\rho_1$ on $\mT'$ for which condition 
{\bf (1)} is satisfied and $L_{\rho_1}(\gamma_{ij}) = q_{ij}\ell_i$, where we recal that the $q_{ij}$'s are rational.

\item We use the strict convexity of the norm $\Norm_\infty$ and 
the fact that $L_{\rho_1}(\gamma_{ij}) = q_{ij}\ell_i$ for $q_{ij} \in \Q$ to show that for any 
minimal cycle $c$ in the length space $(\mG, \rho_1)$ that is not a reparametrization of 
$\gamma_1, \gamma_2, \ldots , \gamma_k$ the following strict inequality holds:
$$L_{\rho_1}(c) > \| (a_c, b_c) \|_{\infty}.$$
Then, since $\|(a_i, b_i)\|_\infty \leq \|(a_{i+1}, b_{i+1})\|_\infty$ for each $i$, we see that the lemma 
is true on the length space $(\mG, \rho_1)$.

\item We use the inequality from the previous step to 
obtain a constant $\Theta >0$ with the property that if $\mT \subset \mT'$ is a regular neighborhood 
of $\mG$ with ``hubs'' $\{\Delta_0, \Delta_1, \ldots , \Delta_t\}$ satisfying 
$$ \diam \Delta_s \leq \max_{x \in \partial \Delta_s} 2 \cdot d(p_s, x) \leq \Theta,$$
where the distance is computed with respect to $\rho_1$,
then the lemma holds on $(\mT, \rho_1)$.
\end{enumerate}
\text{}\\

\noindent
{\bf Step A:} Choose an arbitrary regular neighborhood $\mT'$ of $\mG$ in $T^2$ 
and a collection of ``hubs'' $\{\Delta_0', \Delta_1', \ldots , \Delta_t' \}$, with a corresponding collection of 
rectangular ``corridors'' $\{R_{ij}': 2 \leq i \leq k, 1\leq j \leq m_i\}$,
having the property that 
\begin{eqnarray}\label{eq:HubLength}
L_{g_0}(\Image(\gamma_{ij}) \cap \cup_{s = 0}^{t} \Delta_s) < \frac{1}{2} q_{ij} \ell_{i},
\end{eqnarray}
for each $2\leq i \leq k$, $1 \leq j \leq m_i$.
That is, each edge $\gamma_{ij}$ in our graph $\mG$ has less than $\frac{1}{2}q_{ij}\ell_i$ of its length 
contained in the ``hubs''.
Then by lengthening or shortening each $R_{ij}$ in the ``$\gamma_{ij}$-direction'' we 
obtain a new flat metric $\rho_1$ on $\mT'$
with respect to which we have $L_{\rho_1}(\gamma_{ij}) = q_{ij}\ell_i$ for $2 \leq i \leq k$ $1 \leq j \leq m_i$, 
and (recalling that $\gamma_1$ is trivial) we see that 
$L_{\rho_1}(\gamma_i) = \ell_i,$
for $i = 1, 2, \ldots , k$.
Hence, condition {\bf (1)} of the Lemma is obtained on $(\mT', \rho_1)$.
We note that the fact that $L_{\rho_1}(\gamma_{ij}) = q_{ij}\ell_i$ for each $i = 2, \ldots , k$ and $j = 1, \ldots , m_i$ will be exploited in Step B.\\

We now want to demonastrate that by picking a thin regular 
neighborhood $\mT \subset \mT'$  of $\mG$
where the ``hubs'' can be chosen of suffciently small diameter we can also obtain condition {\bf (2)}.
Towards this end we first show that the lemma is true on our metric graph $(\mG, \rho_1)$.
\text{}\\

\noindent
{\bf Step B:} 
Let $c$ be a minimal cycle in $(\mG, \rho_1)$ representing the homology class $(a_c, b_c) \in H_1(T^2, \Z)$.
Now, for each $i = 2, \ldots, k$ and $j = 1, \ldots , m_i$ we let $n_{ij}^{+}(c)$ (resp. $n_{ij}^{-}(c)$) 
denote the number of times $c$ traverses the edge $\gamma_{ij}$ in the positive direction (respectively, negative direction).
Then with respect to the metric $\rho_1$ the length of the curve $c$ is given by
\begin{eqnarray*} 
L_{\rho_1}(c) &=& \sum_{i=2}^{k}\sum_{j=1}^{m_i} (n_{ij}^{+}(c) + n_{ij}^{-}(c)) L_{\rho_1}(\gamma_{ij}) \\
&=& \sum_{i, j} (n_{ij}^{+}(c) + n_{ij}^{-}(c)) q_{ij} \ell_i \\
&=& \sum_{i, j} (n_{ij}^{+}(c) + n_{ij}^{-}(c)) q_{ij} \| (a_i, b_i) \|_\infty.
\end{eqnarray*}

Since each $q_{ij}$ is rational, we may fix $N \in \N$ so that $Nq_{ij}$ is an integer 
for each $i = 2, \ldots , k$ and $j = 1, \ldots , m_i$. 
Then $c^N$ represents the homology class $N(a_{c}, b_{c})$ and has length
$$L_{\rho_1}(c^N) = \sum_{i= 2}^{k} \sum_{j=1}^{m_i} N(n_{ij}^{+}(c) + n_{ij}^{-}(c)) q_{ij} \| (a_i, b_i) \|_{\infty}.$$
That is, each edge $\gamma_{ij}$ contributes $N(n_{ij}^{+}(c) + n_{ij}^{-}(c)) q_{ij} \| (a_i, b_i) \|_{\infty}$ towards the length of $c^N$.
Now, for each $i =2, \ldots, k$, let $N_i =  \sum_{j=1}^{m_i} N(n_{ij}^{+}(c) - n_{ij}^{-}(c)) q_{ij}$. 
Then $\delta_i \equiv \gamma_i^{N_i}$ is a curve in $\mG$ representing 
the homology class $N_i(a_i, b_i) \in H_1(T^2; \Z)$.
It then follows from the definition of the $n_{ij}^+(c)$'s and $n_{ij}^{-}(c)$'s that the curves $c^N$ and 
$\delta = \delta_{2} * \cdots *  \delta_{k}$ in $\mG$ have the same 
algebraic intersection number with a basis for $H_1(T^2; \Z)$. 
Therefore, since a homology class in $H_1(T^2; \Z)$ is determined by its algebraic intersection numbers with a basis for $H_1(T^2; \Z)$, we conclude that $c^N$ and $\delta$ are homologous, and we obtain the following expression for $N(a_c, b_c)$:
\begin{eqnarray*}
N(a_c, b_c) &=& h_{c^N} \\
		   &=& h_{\delta}\\
		   &=& \sum_{i = 2}^{k} N_i(a_i, b_i)\\
		   &=&  \sum_{i = 2}^{k}\sum_{j=1}^{m_i} N(n_{ij}^{+}(c) - n_{ij}^{-}(c)) q_{ij} (a_i, b_i).
\end{eqnarray*}
The strict convexity of the norm $\Norm_\infty$ allows us to obtain the following:
\begin{eqnarray*}
N\|(a_c, b_c)\|_{\infty} &=& \| N(a_c, b_c)\|_{\infty} \\
				   &=& \| \sum_{i=2}^{k}\sum_{j=1}^{m_i} N (n_{ij}^{+}(c) - n_{ij}^{-}(c)) q_{ij} (a_i, b_i)\|_{\infty} \\
				   &<&  \sum_{i=2}^{k}\sum_{j=1}^{m_i} N |(n_{ij}^{+}(c) - n_{ij}^{-}(c))| q_{ij} \|(a_i, b_i)\|_{\infty} \\
				   &\leq& \sum_{i=2}^{k}\sum_{j=1}^{m_i} N (n_{ij}^{+}(c) + n_{ij}^{-}(c)) q_{ij} \|(a_i, b_i)\|_{\infty} \\
				   &=& NL_{\rho_1}(c).
\end{eqnarray*}
Dividing through by $N$ in the inequality above we obtain
\begin{eqnarray}\label{eq:Inequality}
L_{\rho_1}(c) > \|(a_c, b_c)\|_{\infty}. 
\end{eqnarray}
\text{}\\

As it will be useful in the sequel, we pause to define the notions of hub length and corridor length 
for a loop $\sigma : S^1 \to (\widetilde{\mT}, \widetilde{\rho})$ in an arbitrary flat regular neighborhood of $\mG$. Let $(\widetilde{\mT}, \widetilde{\rho})$ be such a regular neighborhood with a corresponding 
choice of hubs $\{\widetilde{\Delta}_0, \widetilde{\Delta}_1, \ldots , \widetilde{\Delta}_s \}$
and rectangular corridors $\{ \widetilde{R}_{ij} : i = 2, \ldots , k, j = 1, 2, \ldots , m_i \}$. Then the \emph{corridor length} of $\sigma$ is defined to be:
$L_{\tilde{\rho}}^{\rm{hub}}(\sigma) \equiv L_{\tilde{\rho}}(\Image(\sigma) \cap \cup_{s=0}^{t}\widetilde{\Delta}_s).$
Similarly, the \emph{corridor length} of $\sigma$ is defined to be:
$L_{\tilde{\rho}}^{\rm{cor}}(\sigma) = L_{\tilde{\rho}}(\Image(\sigma) \cap \cup_{i,j} \widetilde{R}_{ij})$.
Due to the flatness of the metric $\tilde{\rho}$ and that the corridors are actually rectangles, it is 
clear that for any curve $\sigma$ in $(\tilde{\mT}, \tilde{\rho})$ freely homotopic in $\widetilde{\mT}$ to a minimal cycle $c_\sigma$ in $\mG$ that we have: 
\begin{eqnarray}\label{eq:CorridorLength}
L_{\tilde{\rho}}^{\rm{cor}}(\sigma) &\geq& L_{\tilde{\rho}}^{\rm{cor}}(c_\sigma).
\end{eqnarray}
\text{}\\

\noindent
{\bf Step C:} We are now in a position to explain how to pick our regular neighborhood $\mT \subset \mT'$. 
We begin by defining a particular collection of cycles in our graph $\mG$.

Let $\mC$ denote the collection of \emph{minimal} cycles $c$ in the length space $(\mG, \rho_1)$ with the following properties:

\begin{enumerate}[(1)]
\item $c$ is not freely homotopic to the cycles $\gamma_1, \gamma_2, \ldots , \gamma_k$ in $\mT'$.
(We note that this does not preclude $(a_c, b_c) = (a_i, b_i)$ for some $i =1, 2, \ldots, k$.);
\item $c$ consists of at most $\lfloor \frac{\ell_k}{\zeta} \rfloor$ edges, where 
$\zeta \equiv \frac{1}{2}  \min\{ q_{i1}\ell_i, \ldots , q_{im_i} \ell_i: i = 1, \ldots , k\}$ and 
$\lfloor x \rfloor$ denotes the greatest integer less than $x > 0$. 
(We note that it follows from Equation~\ref{eq:HubLength} and the manner in 
which the metric $\rho_1$ was constructed that the length of each rectangle 
$R_{ij}$ in the ``$\gamma_{ij}$-direction'' is greater than $\zeta$.)
\end{enumerate}
\noindent
It will prove to be useful to notice that the upper bound on edge length of elements of $\mC$ implies that $\mC$ 
is a finite collection. It now follows from Equation~\ref{eq:Inequality} that the quantity 
\begin{eqnarray}\label{eq:Epsilon}
\tilde{\epsilon} \equiv \min_{c \in \mC} (L_{\rho_1}(c) - \|(a_c, b_c)\|_\infty)
\end{eqnarray}
is positive.

Now let $\sigma$ be a curve in the tubular neighborhood $(\mT', \rho_1)$ that is freely homotopic in 
$\mT'$ to the minimal cycle $c_\sigma \in \mC$ of edge length $q \leq \lfloor \frac{\ell_k}{\zeta} \rfloor$. 
The edges of $c_\sigma$ determine $q$ corridors $R_1, \ldots , R_q$ 
through which it passes (counted with multiplicity). Then $\sigma$ must pass 
through these $q$ corridors. In fact, since we are ultimately interested in 
obtaining a lower bound on the length of $\sigma$, we may assume without loss of 
generality that $\sigma$ enter and leaves precisely these $q$ corridors (counting multiplicities) 
and no other corridors. 
As noted earlier, since $(\mT', \rho_1)$ is flat we see that $\Lcor(\sigma) \geq \Lcor(c_\sigma)$.
Hence, the only way that $\sigma$ can be shorter than $c_\sigma$ is to ``make up the difference'' inside the ``hubs''; 
that is, we need the quantity $\Lhub(c_\sigma) - \Lhub(\sigma)$ to be sufficiently large. But, since 
$\Lhub(c_\sigma)$ is bounded from above by 
$$q \cdot \max_{s = 0, 1, \ldots ,t} \max_{x \in \partial \Delta_s'} 2 \cdot d(p_s, x),$$
(where we recall that $p_s$ is the ``center'' of the hub $\Delta_s$)
we have the following crude universal upper bound on the amount any such $\sigma$ can save in 
the hubs compared with its corresponding minimal cycle $c_\sigma$: 
$$\Lhub(c_\sigma) - \Lhub(\sigma) \leq \lfloor \frac{\ell_k}{\zeta} \rfloor \cdot \max_{s = 0, 1, \ldots ,t} \max_{x \in \partial \Delta_s'} 2 \cdot d(p_s, x)$$
Now, suppose we pick a tubular neighborhood $\mT$ of $\mG$ contained in $\mT'$ that is thin enough so that we may choose hubs $\Delta_0, \Delta_1, \ldots , \Delta_t$ satisfying 
$$\lfloor \frac{\ell_k}{\zeta} \rfloor \cdot  \max_{s = 0, 1, \ldots ,t} \max_{x \in \partial \Delta_s} 2 \cdot d(p_s, x) < \frac{\tilde{\epsilon}}{2} < \tilde{\epsilon}.$$
Then for any $\sigma$ in $(\mT, \rho_1)$ freely homotopic in $\mT$ to $c_\sigma \in \mC$ we have 
\begin{eqnarray*}
L_{\rho_1}(\sigma) &=& \Lcor(\sigma) + \Lhub(\sigma)\\
				     &\geq& \Lcor(c_\sigma) + \Lhub(\sigma) \\
				     &=& \Lcor(c_\sigma) + \Lhub(\sigma) + \Lhub(c_\sigma) - \Lhub(c_\sigma) \\
				     &=& L_{\rho_1}(c_\sigma) + \Lhub(\sigma) - \Lhub(c_\sigma) \\
				     &>& L_{\rho_1}(c_\sigma) - \tilde{\epsilon} \\
				     &\geq& \|(a_{c_\sigma}, b_{c_\sigma}) \|_{\infty} \;\; \mbox{(by Equation~\ref{eq:Epsilon})} \\
				     &=& \|(a_\sigma, b_\sigma)\|_{\infty}.
\end{eqnarray*} 
\noindent
In particular, if $(a_\sigma, b_\sigma) \neq (a_i, b_i)$ for some $i =1, \ldots , k$, then
$$L_{\rho_1}(\sigma) >  \|(a_\sigma, b_\sigma)\|_{\infty} \geq \ell_k.$$

Now let $\sigma$ be a loop in $(\mT, \rho_1)$ that is freely homotopic in 
$\mT$ to a minimal cycle $c_\sigma \not\in \mC$.
Then $c_\sigma$ can be taken to be $\gamma_i$ for some $i = 2, \ldots , k$ or 
$c_\sigma$ has $q$ edges where $q \geq \lfloor \frac{\ell_k}{2} \rfloor + 1$.
In the former case, since we are once again interested in a lower bound 
on the length of $\sigma$ we can assume without loss of generality that 
$\sigma$ is contained in a (flat) tubular neighborhood $\mT_i \subset \mT$ 
of $c_\sigma \equiv \gamma_i$. 
But, then it follows that since $\rho_1$ is flat that we have 
$$L_{\rho_1}(\sigma) \geq L_{\rho_1}(\gamma_i) = \ell_i = \|(a_i, b_i)\|.$$
In the latter case, we see that $\sigma$ must pass through at least $q$ corridors. 
Then since each corridor is of length at least $\zeta$ we see
$$L_{\rho_1}(\sigma) \geq \zeta q \geq \zeta \cdot (\lfloor \frac{\ell_k}{\zeta} \rfloor + 1) > \ell_k.$$
\text{}\\

In summary, consider the flat regular neighborhood $(\mT', \rho_1)$ of $\mG$ constructed in Step A 
and choose a regular neighborhood $\mT \subset \mT'$ of $\mG$ with 
``hubs'' $\{\Delta_0, \Delta_1, \ldots , \Delta_s\}$ satisfying 
$$\diam \Delta_s \leq \max_{x \in \partial \Delta_s} 2\cdot d(p_s, x) \leq \Theta \equiv \frac{\tilde{\epsilon}}{2\lfloor \frac{\ell_k}{\zeta} \rfloor},$$
for each $s = 0, 1, \ldots , t$, as in Step C. If $\sigma$ is a loop in $(\mT, \rho_1)$, then
\begin{itemize}
\item $L_{\rho_1}(\sigma) \geq \ell_k$, if $(a_\sigma, b_\sigma) \neq \pm (a_1, b_1), \ldots , \pm (a_k, b_k)$;
\item $L_{\rho_1}(\sigma) \geq \ell_i$, if $(a_\sigma, b_\sigma) = \pm(a_i, b_i)$ for some $i = 1, 2, \ldots , k$;
\item $L_{\rho_1}(\gamma_i) = \ell_i$ for each $i = 1, 2, \ldots, k$.
\end{itemize} 
\end{proof}
\text{}\\

Now let $\mT_0 \subset \mT_1 \subset \cdots \subset \mT_4 = \mT$ be a collection of 
properly nested tubular neighborhoods of the graph
$\mG$ with smooth boundaries such that 
\begin{enumerate}[(1)]
\item $\mT = \mT_4$ admits a metric $\rho_1$ as in the lemma;
\item for each $i = 1, \ldots , 4$ and $p, q \in \partial \mT_i$ we have $
d(p, \partial \mT_0) = d(q, \partial \mT_0)$ where the distance is taken with respect to the metric $\rho_1$
\end{enumerate}
and let $\Gamma_i = d(\partial \mT_i, \partial \mT_0)$ for each $i = 1, \ldots , 4$.
Now define the smooth function $r: T^2 \to \R$ via
$$r(x) = \left\{ \begin{array}{ll}
0 & \mbox{ for } x \in \overline{\mT_0} \\
d_{\rho_1}(x, \partial \mT_0) & \mbox{oterhwise}
\end{array}
\right.
$$
Now let $\kappa >0$ be such that with respect to $\kappa \rho_1$ the distance between 
$\partial \mT_2$ and $\partial \mT_1$ is at least $B$.

\begin{lemma}[cf. Lemma 5.3 of \cite{DGS}]\label{lem:TheMetric}
With the notation and assumptions above there is a Riemannian metric $g$ on $T^2$ with the following properties:
\begin{enumerate}[\bf (1)]
\item $g \succeq g_0$ on $T^2 - \mT_1$;
\item $g = g_0$ on $T^2 - \mT_3$;
\item $g \succeq \kappa \rho_1$ on $\mT_2 - \mT_1$;
\item $g = \rho_1$ on $\mT_0$;
\item $g \succeq \rho_1$ on $\mT_2$;
\end{enumerate}
where for any metrics $\rho$ and $\tilde{\rho}$ we write $\rho \succeq \tilde{\rho}$ 
if for all vectors $v$ we have $\rho(v,v) \geq \tilde{\rho}(v,v)$.
\end{lemma}

\begin{proof}[Proof of Lemma~\ref{lem:TheMetric}]
The proof is exactly the same as in \cite[Lemma 5.3]{DGS}, but we include it for completeness. 
First, consider the metric $\rho_2 = g_0 + \kappa \rho_1$ on $\mT = \mT_4$.
This metric clearly satisfies $\rho_2 \succeq g_0$ on $\mT$.
Now let $f_1: [0, \Gamma_4] \to [0,1]$ be a smooth function such that 
$f_1(t) = 1$ for $0 \leq t \leq \Gamma_2$ and $f_1(t) = 0$ for $\Gamma_3 \leq t \leq \Gamma_4$.
We now define a metric $\widehat{g}$ on $T^2$ as follows:
$$\widehat{g} = \left\{ 
\begin{array}{ll}
(f_1 \circ r)\rho_2 + (1-(f_1 \circ r))g_0 & \mbox{ on } \mT \\
g_0 & \mbox{on } M - \mT_3
\end{array}
\right.
$$
Then on $T^2$ we have $\widehat{g} \succeq g_0$ and on $\mT_3$ we have 
$\widehat{g} = \rho_2 \succeq \kappa \rho_1 \succeq \rho_1$.
Now let $f_2 : [0, \Gamma_4] \to [0,1]$ be a smooth function such that 
$f_2(0) = 1$ and $f_2(t) = 0$ for $\Gamma_1 \leq t \leq \Gamma_4$ 
and set 
$$g = \left\{
\begin{array}{ll}
(f_2 \circ r)\rho_1 + (1-(f_2 \circ r))\widehat{g} & \mbox {on } \mT \\ 
\widehat{g}  & \mbox{on } T^2 - \mT_1
\end{array}
\right.
$$
Then $g$ satisfies properties $(1)$-$(5)$.
\end{proof}

We now show that any metric $g$ on $T^2$ as in Lemma~\ref{lem:TheMetric} has the desired properties.
Indeed, let $g$ be  such a metric and let $\sigma$ be a homotopically non-trivial curve in $T^2$.
Then there are three cases.

\medskip
\noindent
{\bf Case A:} $\Image (\sigma) \subset \mT_2 - \mT_1$. 

\medskip
Then by Lemma~\ref{lem:TheMetric}(1) 
$g \succeq g_0$ on $T^2 - \mT_1$, so we see 
$$L_g(\sigma) \geq L_{g_0}(\sigma) \geq \Syst(T^2, g_0) = B \geq \ell_k.$$

\medskip
\noindent
{\bf Case B:} $\Image(\sigma) \cap \mT_1 \neq \emptyset$ and $\Image(\sigma) \cap (T^2 -\mT_2) \neq \emptyset$.

\medskip
Then, by Lemma~\ref{lem:TheMetric}(3) and the way in which $\kappa$ was chosen, we see
$$L_g(\sigma) \geq d_{g}(\partial \mT_2, \partial \mT_1) \geq B \geq \ell_k.$$

\medskip
\noindent
{\bf Case C:} $\Image(\sigma) \subset \mT_2$.

\medskip
If $(a_\sigma, b_\sigma) \neq (a_1, b_1), \ldots , (a_k, b_k)$, then using 
Lemma~\ref{lem:TheMetric}(5) and Lemma~\ref{lem:TubNbhd} we see
$$L_g(\sigma) \geq L_{\rho_1}(\sigma) \geq \ell_k.$$
If $(a_\sigma, b_\sigma) = (a_i, b_i)$ for some $i = 1, \ldots , k$ then using (5) and (4) of Lemma~\ref{lem:TheMetric} we see
$$L_g(\sigma) \geq L_{\rho_1}(\sigma) \geq \ell_i.$$

\medskip 
We complete the proof of Theorem~\ref{thm:StableNorm}(i) by noting that Lemma~\ref{lem:TheMetric}(4) and Lemma~\ref{lem:TubNbhd}
imply that for each $i =1, \ldots, k$ $L_g(\gamma_i) = L_{\rho_1}(\gamma_i) = \ell_i$.

\subsection{The proof of Theorem~\ref{thm:MinLengthSpec}}

Let $\Psi: \pi_1(T^2, p_0) \to H_1(T^2; \Z)$ denote the Hurewicz isomorphism and notice that 
for any $h \in H_1(T^2; \Z) = H_1(T^2; \Z)_\R$ we have $\| h \|_s = m_g(\pi(\Psi^{-1}(h)))$, 
where $\pi: \pi_1(T^2) \to \mF(T^2)$ is the natural projection of the fundamental group of $T^2$ 
onto the collection of its unoriented free homotopy classes (see p.~\pageref{def:NatProj}). 
It is then apparent that the statement ``{\bf (ii)} implies {\bf (i)}'' is actually a reformulation 
of Bangert's observation that the stable norm of a $2$-torus is strictly convex and 
the statement ``{\bf (i)} implies {\bf (ii)}'' is equivalent to Theorem~\ref{thm:StableNorm}{\bf (i)}.
This completes the proof.

\section{Convergence of the Stable Norms: the Proof of Theorem~\ref{thm:StableNorm}(ii)}\label{sec:Part2}

In this section we demonstrate that the sequence $\langle \Norm_j \rangle_{j \in \N}$ of 
toral stable norms constructed in the previous section converge in the topology of compact convergence 
to the fixed strictly convex norm $\Norm_\infty$.

Let $g$ be a metric on $T^2$ and as in the introduction for each $h \in H_1(T^2; \Z)_{\R}$ let
$$f(h) = \inf \{L_g(\sigma) : \sigma \mbox{ is a smooth curve representing the class } h \}.$$
Then we have:
\begin{enumerate}[(1)]
\item $f(h_1 + h_2) \leq f(h_1) + f(h_2)$ for any $h_1, h_2 \in H_1(T^2; \Z)_{\R}$,
\item $f(k h) = |k|f(h)$ for any $h \in H_1(T^2; \Z)_{\R}$ and $k \in \Z$; in particular, $f(-h)=f(h)$.
\end{enumerate}
From this we can conclude that for each $h_1, h_2 \in H_1(T^2; \Z)_{\R}$ 
we have $|f(h_1) - f(h_2)| \leq f(h_1- h_2) = f(h_2-h_1)$. It then follows that the associated stable norm 
$\Norm_s$ will have the property that 
$$| \|x\|_s - \|y\|_s | \leq \|x-y\|_s$$
for each $x, y \in H_1(T^2; \R) \simeq \R^2$.
Now we recall the following basic fact about norms on finite dimensional vector spaces. 

\begin{lemma}[cf. Theorem 7.7 \cite{Dym}]\label{lem:EquivNorm}
Let $\phi, \psi : \R^n \to \R$ be norms, then there are constants $0 < A \leq B$ such that
$$A \psi (x) \leq \phi(x) \leq B \psi(x).$$
In fact, $A$ and $B$ can be taken to be 
$$A = \frac{\inf \{\phi(x) :|x| = 1\}}{\left(\sum_{i = 1}^{n} \psi(\vece_i)^2\right)^{1/2}}$$
and 
$$B = \frac{\left(\sum_{i = 1}^{n} \phi(\vece_i)^2\right)^{1/2}}{\inf \{\psi(x) : |x| = 1\}  },$$
where $\{ \vece_1, \ldots , \vece_n \}$ is the standard basis for $\R^n$ and 
$| \cdot |$ denotes the standard Euclidean norm with respect to this basis.
\end{lemma}

\begin{proof}
Since $\phi$ is a norm we see that for any $x, y \in \R^n$ 
$$\phi(x) = \phi(x-y + y) \leq \phi(x-y) + \phi(y).$$
It then follows from the fact that $\phi(-v) = \phi(v)$ that 
$$|\phi(x) - \phi(y)| \leq \phi(x-y).$$
Now, let $\{ \vece_1, \ldots , \vece_n \}$ be the standard basis for $\R^n$ and let 
$x = \sum_{i =1}^{n} x_i \vece_i$ and $y = \sum_{i=1}^{n} y_i \vece_i$ be vectors in $\R^n$.
Then 
\begin{eqnarray*}
|\phi(x) - \phi(y) | &\leq& \phi(x-y) \\
			      &=& \phi(\sum_{i=1}^{n} (x_i - y_i) \vece_i ) \\
			      &\leq& \sum_{i=1}^{n} |x_i - y_i| \phi(\vece_i)\\
			      &\leq& |x-y||\sum_{i=1}^{n} \phi(\vece_i) \vece_i| \mbox{ (by the Cauchy-Schwarz Inequality)}
\end{eqnarray*}
where $| \cdot |$ denotes the usual Euclidean norm. Hence, $\phi$ is continuous and 
taking $y$ to be zero in the equation above we obtain
$$ \phi(x) = |\phi(x)| \leq |x|(\sum_{i=1}^{n} \phi(\vece_i)^2),$$
for each $x \in \R^n$. It is then clear that 
\begin{eqnarray}\label{eq:One}
|x| \inf \{ \phi(v) : |v | =1 \} \leq \phi(x) \leq |x| \sum_{i=1}^{n} \phi(\vece_i)^2,
\end{eqnarray}
for any $x \in \R^n$. 
Similarly, we see that $\Psi$ is continuous and that for each $x \in \R^n$ 
\begin{eqnarray}\label{eq:Two}
|x| \inf \{ \psi(v) : |v | =1 \} \leq \psi(x) \leq |x| \sum_{i=1}^{n} \psi(\vece_i)^2.
\end{eqnarray}
Since $\psi$ is continuous, we see that $\inf \{ \psi(v) : |v | =1 \} $ is positive. 
Therefore, we may combine Equations~\ref{eq:One} and \ref{eq:Two} to establish the claim. 
\end{proof}

Hence, we see that for any  $x , y \in \R^2$ and stable norm $\Norm_{s}$ on the $2$-torus we have 
\begin{eqnarray*}
| \|x\|_s - \|y\|_s | &\leq& \|x-y\|_{s} \\
				&\leq&  \frac{(\sum_{i = 1}^{2} \| \vece_i \|_s^2)^{1/2}}{\inf \{ |v | : \sum_{i=1}^{2} v_i^2 = 1\}  } |x-y| \\
				&=& (\sum_{i = 1}^{2} \| \vece_i \|_s^2)^{1/2}|x-y|
\end{eqnarray*}
where throughout $| \cdot |$ denotes the standard Euclidean norm on $\R$ and $\R^2$.\\

We now turn our attention to the sequence of stable norms $\langle \Norm_j \rangle_{j = 1}^{\infty}$
converging to the stable norm $\Norm_\infty$ given by Theorem~\ref{thm:StableNorm}(a). 
The $\Norm_{k}$'s were constructed so that 
for each $k \in \N$ we have $\| (a_j, b_j)\|_k = \|(a_j, b_j)\|_\infty$ for any $1\leq j \leq k$ 
and $\|(a_j, b_j)\|_k \geq \|(a_k, b_k)\|_\infty$ for all $j \geq k+1$.
Now fix $N$ large enough so that $(1,0)$ and $(0,1)$ are among the vectors 
$\{(a_j, b_j) : 1 \leq j \leq N\}$. 
Then we see that  for each $j \geq N$ we have  $\| (1,0) \|_j \equiv \| (1,0) \|_\infty$ and $\| (0,1) \|_j \equiv \| (0,1) \|_\infty$, 
and it follows from Lemma~\ref{lem:EquivNorm} that the constant  
$B = (\| (1,0) \|_\infty^2 + \| (0,1) \|_\infty^2)^{1/2} $ satisfies
$$ |\| x\|_j - \|y \|_j | \leq B |x-y|$$
for each $j \geq N$.
That is, for $j \geq N$ the stable norms $\Norm_j$ are Lipschitz 
continuous with the same Lipschitz constant $B$.
We now recall the following fact about Lipschitz continuous functions on $\R^n$.

\begin{lemma}\label{lem:CptConv}
Let $\langle f_j \rangle_{j =1}^{\infty}$ be a sequence of functions on $\R^n$
for which there exists a constant $C \geq 0$ such that for each $n$
$$| f_n(x) - f_n(y) | \leq C |x-y| \mbox{ for all } x, y \in \R^n.$$
(That is, $\langle f_n\rangle_{n=1}^{\infty}$ is a sequence of Lipschitz continuous functions with the same Lipschitz constant $C$.)
If the sequence $\langle f_n\rangle_{n =1}^{\infty}$ converges pointwise to $f: \R^n \to \R$, then 
$f = \lim_{n \to \infty} f_n$ in the topology of compact convergence. 
\end{lemma}

\begin{proof}
We will show that the $f_j$'s form a uniformly Cauchy sequence on any compact subset $K$ of $\R^n$. 
That is, given $\varepsilon > 0$ and compact subset $K \subset \R^n$, there is an $N \in \N$ such that $|f_n(y) - f_m(y)| < \varepsilon$ 
for all  $n, m \geq N$ and $y \in K$. This implies that $f_j \to f$ uniformly on $K$.

Fix $\varepsilon >0$. Then since $\langle f_j \rangle_{j =1}^{\infty}$ is a sequence of Lipschitz continous functions with the same Lipschitz constant C we see that for any $j \in \N$ we have $|f_j(x) - f_j(y) | < \frac{\varepsilon}{3}$, 
when $|x-y| < \delta \equiv \frac{\varepsilon}{3C}$. Now, since the $f_j$'s converge pointwise to $f$ we see that
for any $x \in \R^n$ there is an $N_x \in \N$ such that $|f_n(x) - f_m(x)| < \frac{\varepsilon}{3}$ for all $n, m \geq N_x$. It then follows that 
for any $y \in B(x, \delta)$ we have for each $n, m \geq N_x$
\begin{eqnarray*}
|f_n(y) - f_m(y)| &\leq& |f_n(x) - f_n(x)| + |f_n(x) - f_m(x)|+|f_m(x) - f_m(y)| \\
			     &\leq& \frac{\varepsilon}{3} + \frac{\varepsilon}{3} + \frac{\varepsilon}{3} \\
			     &=& \varepsilon.
\end{eqnarray*} 
Now let $K$ be compact subset of $\R^n$, then there are $x_1, \ldots , x_q \in K$ 
such that $K \subset \cup_{i=1}^{q} B(x_i, \delta)$. Taking $N = \max\{N_{x_1}, \ldots N_{x_q}\}$ it follows that for any 
$y \in K$ and $n,m \geq N$ we have $|f_n(y) - f_m(y) | < \varepsilon$.
\end{proof}

Now, by design, the sequence $\langle \Norm_j \rangle_{j =1}^{\infty}$ converges pointwise to $\Norm_\infty$ 
on the rational points in $\R^2$, but by continuity and denseness we see that they converge pointwise on all of $\R^2$ to $\Norm_\infty$. Since for all $j \geq N$, the stable norm $\Norm_j$ is Lipschitz continuous with Lipschitz constant 
$B = (\| \vece_1\|_\infty^2 + \| \vece_2\|_\infty^2)^{1/2} $, it follows from the previous lemma 
that $\lim_{j \to \infty} \Norm_j = \Norm_\infty$ in the topology of compact convergence.
This completes our argument.

\section{Multiplicities in the minimum marked length spectrum of tori}\label{sec:Multiplicity}

In this section we will prove Theorem~\ref{thm:Multiplicity} which tells us that if $\ell$ is a 
length of multiplicity $m$ in the minimum length spectrum of $(T^2, g)$, then 
$n \equiv \# m_g^{-1}([0, \ell))$ is bounded from below by a function $f(m)$. 
That is, if we wish to find a length of multiplicity $m$ in the minimum length spectrum,
then we must look beyond the $f(m)$-th term of the sequence.
Before proving this theorem it will be useful to recall some facts concerning integer $n$-gons in $\R^2$.

An \emph{integer $n$-gon} is an $n$-gon  in $\R^2$ whose vertices lie in the lattice $\Z^2$.
Given an integer $n$-gon $P$, Pick's theorem tells us that the area of the region bounded by $P$, denoted by $A(P)$,  
can be computed as follows
$$A(P) = i(P) + \frac{b(P)}{2} -1,$$
where $i(P)$ denotes the number of lattice points in the interior of the region bounded by $P$ and $b(P)$ equals the number of 
lattice points on the boundary $P$.
Now for each $k$ we let $\mathcal{P}_k^{+}$ denote the collection of convex integer $k$-gons and set
$$A(k) \equiv \min\{ A(P) : P \in \mathcal{P}_k^{+} \}.$$
Bounds for the function $A(k)$ have been studied for some time. 
In fact, Andrews was the first to observe that $A(k)$ grows roughly like $k^3$ \cite{Andrews}. 
Some recent improvements and related results that will be useful in our argument are summarized below.
But, first we give a definition.

\begin{dfn}
A (bounded) region $R \subset \R^2$ is said to be \emph{centrally symmetric} with respect to $p_0 \in \R^2$, 
if for any $x_1 \in R$ there is a point $x_2 \in R$ such that $p_0$ is the midpoint of the line segment joining 
$x_1$ and $x_2$.
\end{dfn}

\begin{proposition}\label{prop:A(k)}
The function $A(k)$ enjoys the following properties:

\begin{enumerate}[{\bf (1)}]
\item (Rabinowitz, \cite{Rab}) $\frac{1}{8 \pi^2} < \frac{A(k)}{k^3} < \frac{1}{54} + O(1)$;
\item (B\'{a}r\'{a}ny-Tokushige, \cite[Theorem 1]{BT}) $\lim_{k \to \infty} \frac{A(k)}{k^3}$ exists;
\item (B\'{a}r\'{a}ny-Tokushige, \cite[Claim 1]{BT}) for every $k$ even, there is a $k$-gon $\widehat{P}_k$ that is 
centrally symmetric with respect to some $(x,y) \in \frac{1}{2} \Z^2$ and such that $A(k) = A(\widehat{P}_k)$.
\end{enumerate}
\end{proposition}

Now for each $k \in \N$ we let 
$$i(k) \equiv \min \{i : \mbox{ there exists a $k$-gon with exactly $i$ interior points} \}.$$
Then since any integer $k$-gon contains an inscribed $k$-gon with exactly $k$ lattice points on it,
Pick's theorem tells us that 
$$i(k) = A(k) + \frac{2-k}{2},$$
and we conclude that the problem of finding among convex $k$-gons the least number of 
interior points is the same as finding the $k$-gon of smallest area. 
The following is than an immediate consequence of Proposition~\ref{prop:A(k)}.

\begin{proposition}\label{prop:i(k)}
The function $i(k)$ enjoys the following properties:

\begin{enumerate}[{\bf (1)}]
\item $\frac{1}{8 \pi^2} + o(k) < \frac{i(k)}{k^3} < \frac{1}{54} + O(k)$;
\item $\lim_{k \to \infty} \frac{i(k)}{k^3}$ exists;
\item for every $k$ even, there is an $k$-gon $\widehat{P}_k$ that is 
centrally symmetric about the origin $(0,0)$ and such that $i(k) = i(\widehat{P}_k)$.
\end{enumerate}
\end{proposition}

With these preliminaries out of the way we may now prove Theorem~\ref{thm:Multiplicity}.

\begin{proof}[Proof of Theorem  \ref{thm:Multiplicity}]

Consider a torus $(T^2, g)$ which has a length $\ell$ of multiplicity $m$ in its minimum length spectrum, 
and let $\pm(a_1, b_1), \ldots , \pm(a_m, b_m) \in \Z^2$ represent the (not necessarily primitive) 
unoriented free homotopy classes in $m_g^{-1}(\ell)$, and set $n \equiv \# m_g^{-1}([0,\ell))$.
Then the points $\pm(a_1, b_1), \ldots , \pm(a_m, b_m)$ determine an  integer $2m$-gon that
is centrally symmetric about $(0,0)$ with exactly $2n-1$ interior points.
Now, for each $k \in \N$ we consider the following odd integer 
$$i_0^{\rm{symm}}(2k) \equiv \{i(Q) : Q \in \widehat{P}_{2k} \mbox{ is centrally symmetric with respect to $(0,0)$}  \}.$$
Then we see that 
$$n \geq f(m) \equiv \frac{ i_0^{\rm{symm}}(2m) + 1}{2}.$$
Since $ i_0^{\rm{symm}}(2m) \geq i(2m) = O(m^3)$, this establishes the first part of the claim.

To see that this inequality is sharp pick $m \in \N$ and let $\widehat{Q}_{2m}$ be a $2m$-gon with vertices 
$\{ \pm(a_1, b_1), \ldots , \pm(a_m, b_m)\}$ that is centrally symmetric with 
respect to $(0,0)$ and such that $ i(\widehat{Q}_{2m}) =  i_0^{\rm{symm}}(2m)$.
Since every centrally symmetric $2m$-gon with center $(0,0)$ contains an inscribed centrally symmetric $2m$-gon with 
center $(0,0)$ and whose only intersection with $\Z^2$ occurs at the $2m$-vertices, we see that 
the boundary of $\widehat{Q}_{2m}$ contains exactly the $2m$ vertices.
Now let $c$ be the boundary of a strictly convex region $B \subset \R^2$ that is centrally symmetric with respect to $(0,0)$
and such that the intersection of $c$ with $\Z^2$ is precisely the collection of vertices of $\widehat{Q}_{2m}$.
(There are many ways to find such a curve. One way is by replacing each of the segments in $\widehat{Q}_{2m}$ by convex polynomial arcs such that the resulting tangent vectors at the beginning and the end of the arcs remain outside the resulting shape. The centrally symmetric condition is easily met by doing this simultaneously on opposite edges with symmetric arcs.)

Now let $\Norm_c$ be the unique strictly convex norm on $\R^2$ such that $c$ is precisely the set of points in $\R^2$ with $\|(x,y) \|_c = \ell > 0$.
Then by Theorem~\ref{thm:StableNorm} there is a Riemannian metric $g$ on $T^2$ whose stable norm agrees with $\Norm_c$ on the set of 
$(a,b) \in \Z^2$ such that $\|(a,b) \|_c \leq \ell$ and has norm strictly larger than $\ell$ for all other lattice points. 
It follows that the metric $g$ is such that $\ell$ has multiplicity $m = \# m_g^{-1}(\ell)$ in the minimum length spectrum and 
the number of unoriented free homotopy classes for which the shortest geodesic is of length less than $\ell$ is 
precisely $\# m_g^{-1}([0, \ell) ) = f(m) \equiv \frac{ i_0^{\rm{symm}}(2m) + 1}{2}$.
\end{proof}

\begin{proof}[Proof of Corollary~\ref{cor:HypMult}]
The idea here is to take a hyperbolic once punctured or one holed torus, construct 
a comparable compact smooth closed torus from it, and apply Theorem \ref{thm:Multiplicity}. More precisely, for a one-holed torus with geodesic boundary, by glueing in a euclidean hemisphere of the same boundary length, one obtains a closed torus. For a once-punctured torus, one can mimic this construction by first removing a small horocyclic neighborhood of the cusp, of say length $1$, and then glueing a euclidean hemisphere of equator length $1$. Minimum length geodesics on this torus do not enter the added euclidean hemisphere. To see this, consider an arc of a curve that does cross a hemisphere. The arc has length at least the length of the shortest equator path between the two endpoints of the arc. The new curve obtained by replacing the arc by the equator path is either shorter or of equal length but is no longer smooth and thus cannot be of minimum length. We can conclude that a minimum length geodesic is entirely contained in the hyperbolic part of the torus. As minimum length curves are always simple closed geodesics, the result on minimum length curves on a smooth torus now naturally correspond to simple closed geodesics on the hyperbolic tori. Now Theorem~\ref{thm:Multiplicity} asserts that if there are $m$ distinct homotopy classes associated to equal minimum length geodesics, then there are at least $f(m)$ homotopy classes with shorter length representatives and this proves the corollary.
\end{proof}

\begin{remark}
We note that in \cite{MR}, McShane and Rivin used the stable norm on the homology of a punctured torus to 
study the asymptotic growth of the number of simple closed geodesics of length less than $\ell$ on a hyperbolic torus.
\end{remark}

\bibliographystyle{amsalpha}
\def\cprime{$'$}

\vspace{0.2cm}

\end{document}